\documentclass[12pt]{amsart}
\usepackage{amssymb}
\theoremstyle{plain}
 \newtheorem{thm}{Theorem}[section]
 
 \newtheorem{lem}[thm]{Lemma}
 \newtheorem{prop}[thm]{Proposition}

\theoremstyle{definition}

\theoremstyle{remark}

\begin{document}
\title[on a norm inequality for a positive block-matrix.]
{on a norm inequality for a positive block-matrix.}

\author[Tomohiro Hayashi]{{Tomohiro Hayashi} }
\address[Tomohiro Hayashi]
{Nagoya Institute of Technology, 
Gokiso-cho, Showa-ku, Nagoya, Aichi, 466-8555, Japan}

\email[Tomohiro Hayashi]{hayashi.tomohiro@nitech.ac.jp}

\baselineskip=17pt

\maketitle

\begin{abstract}
For a positive semidefinite matrix 
$H=
\begin{bmatrix}
A&X\\
X^{*}&B
\end{bmatrix}
$, we consider the norm inequality 
$
||H||\leq  
||A+B||
$. We show that this inequality holds under certain conditions. 
Some related topics are also investigated. 
\end{abstract}

\section{Introduction}

In this paper we investigate the following problems posed by Minghua Lin \cite{L2}. 

\noindent{\bf Problem 1}. 
Let $A>0$, $B>0$ and $X$ be matrices satisfying 
$H=
\begin{bmatrix}
A&X\\
X^{*}&B
\end{bmatrix}\geq 0
$, or equivalently 
$B\geq X^{*}A^{-1}X$. Under what condition 
can we conclude 
$
||H||\leq  
||A+B||
$?  

\noindent{\bf Problem 2}. 
Let $A>0$ and $X$ be matrices. 
Under what condition 
can we conclude  
$$
||A+A^{-\frac{1}{2}}XX^{*}A^{-\frac{1}{2}}||\leq  
||A+X^{*}A^{-1}X||?
$$  
As explained later the problem 2 is a special case of the problem 1. 
It is shown in \cite{BL}\cite{LW1} 
that if $X=X^{*}$, then the inequality 
in the problem 1 holds. Hiroshima \cite{H} showed that 
if we have both  
$
\begin{bmatrix}
A&X\\
X^{*}&B
\end{bmatrix}\geq 0
$ and 
$
\begin{bmatrix}
A&X^{*}\\
X&B
\end{bmatrix}\geq 0
$, then the inequality 
in the problem 1 is true. 
(See also \cite{L1} and \cite{LW2} for more information on this topic.) 
Related to these problems, Lin conjectured the following. 

\noindent{\bf Lin's conjecture (\cite{L2} 
See also \cite[Conjecture 2.14]{BLL} ).} 
If $X$ is normal and all matrices are $2\times2$, 
then the inequality in the problem 1 is true. 

Lin showed that this conjecture is OK in the case that 
$B=X^{*}A^{-1}X$. More generally it is shown in \cite{M} that 
the inequality in the problem 1 is true in the case that the numerical range 
of $X$ is a line segment. 
Here we should remark that in this case $X$ must be normal. On the 
other hand if $X$ is $2\times2$ and normal, then its numerical range 
is a line segment.  Thus we know that 
Lin's conjecture is true. 
In \cite{BM} Bourin and Mhanna generalized this theorem. 

There are two main results in this paper. The first one is a partial answer to the 
problem 1. We show that if  $||A+X^{*}A^{-1}X||\geq||A+XA^{-1}X^{*}||$ then the inequality 
$||H||\leq||A+B||$ is true. Moreover we show that 
if  $||A+X^{*}A^{-1}X||\leq||A+XA^{-1}X^{*}||$ then we have 
$$\Bigl\Vert
\begin{bmatrix}
A&X^{*}\\
X&C
\end{bmatrix}
\Bigr\Vert\leq ||A+C||
$$ for any $C\geq XA^{-1}X^{*}$. 
The second main result is as follows. We expect that if the inequality in the problem 1 
holds for any $A>0$ and $B>0$ with $B\geq X^{*}A^{-1}X$, then $X$ must be normal. 
We show that this is true if the eigenvalues of $XX^{*}$ are distinct.  

The author wishes to express his hearty gratitude to 
Professor Minghua Lin 
for his kind explanation and advice. The author also would like to 
thank  
Professors J.-C. Bourin and Antoine Mhanna 
for valuable comments.

\section{Main results.}

Throughout this paper we consider $n\times n$-matrices acting on ${\Bbb C}^{n}$. 
We denote by $||A||$ the operator norm of the matrix $A$. That is, 
$||A||^2$ is the maximal eigenvalue of $A^{*}A$. 
For two vectors $\xi,\eta\in{\Bbb C}^{n}$ their inner product is denoted by 
$\langle \xi,\eta\rangle$. 
We define the norm of the vector $\xi\in{\Bbb C}^{n}$ by 
$||\xi||=\langle\xi,\xi\rangle^{\frac{1}{2}}$. 
The matrix $A$ is called positive semidefinite if 
$\langle A\xi,\xi\rangle\geq0$ for any $\xi\in{\Bbb C}^{n}$ and we use the notation 
$A\geq0$. We also use the notation $A>0$ if $A\geq0$ and $A$ is invertible. For 
two self-adjoint matrices $A$ and $B$ the order $A\leq B$ is defined by $B-A\geq0$.

At first we give some remarks on the probelms. 
\begin{enumerate}
\item 
The problem 2 is a special case of the problem 1. Indeed we have 
$
\begin{bmatrix}
A&X\\
X^{*}&X^{*}A^{-1}X
\end{bmatrix}\geq 0
$ and 
\begin{align*}
\Bigl\|&
\begin{bmatrix}
A&X\\
X^{*}&X^{*}A^{-1}X
\end{bmatrix}
\Bigr\|
=
\Bigl\|
\begin{bmatrix}
A^{\frac{1}{2}}\\
X^{*}A^{-\frac{1}{2}}
\end{bmatrix}
\begin{bmatrix}
A^{\frac{1}{2}}&
A^{-\frac{1}{2}}X
\end{bmatrix}
\Bigr\|\\
&=
\Bigl\|\begin{bmatrix}
A^{\frac{1}{2}}&
A^{-\frac{1}{2}}X
\end{bmatrix}
\begin{bmatrix}
A^{\frac{1}{2}}\\
X^{*}A^{-\frac{1}{2}}
\end{bmatrix}
\Bigr\|=||A+A^{-\frac{1}{2}}XX^{*}A^{-\frac{1}{2}}||.
\end{align*}

\item 
In general the inequality in the problem 2 does not hold.  
Assume that the inequality in the problem 2 
is true for any matrices $A>0$ and $X$. 
By this assumption we have 
$$
||A+A^{-\frac{1}{2}}XX^{*}A^{-\frac{1}{2}}||\leq  ||A+X^{*}A^{-1}X||
$$ 
We also have 
\begin{align*}
||
A+A^{-\frac{1}{2}}&
(A^{\frac{1}{2}}X^{*}A^{-\frac{1}{2}})(A^{\frac{1}{2}}X^{*}A^{-\frac{1}{2}})^{*}
A^{-\frac{1}{2}}
||\\ 
&\leq 
||
A+(A^{\frac{1}{2}}X^{*}A^{-\frac{1}{2}})^{*}A^{-1}(A^{\frac{1}{2}}X^{*}A^{-\frac{1}{2}})
||
\end{align*} 
and hence 
$$
 ||A+X^{*}A^{-1}X||\leq ||A+A^{-\frac{1}{2}}XX^{*}A^{-\frac{1}{2}}||.
$$
Therefore we conclude that 
$$
||A+A^{-\frac{1}{2}}XX^{*}A^{-\frac{1}{2}}|
=||A+X^{*}A^{-1}X||.
$$ 
Consider the case $X=U$ unitary. Then we have 
$$
||A+A^{-1}||
=||A+U^{*}A^{-1}U||.
$$
If $A$ is a $3\times3$ diagonal matrix and $U$ is a permutation matrix, 
then it is easy to construct a counterexample. 
\end{enumerate}

The following is a key lemma for our investigation. 

\begin{lem}
Let $A>0$, $B>0$ and $X$ be matrices. 
(We don't  have to assume $B\geq X^{*}A^{-1}X$.) 
If $\Bigl\|
\begin{bmatrix}
A&X\\
X^{*}&B
\end{bmatrix}
\Bigr\|
>||A+B||
$, 
then we have 
$$
||A+XA^{-1}X^{*}||\geq 
\Bigl\|
\begin{bmatrix}
A&X\\
X^{*}&B
\end{bmatrix}
\Bigr\|>||A+B||.
$$ 
In particular if $B\geq X^{*}A^{-1}X$, then we have 
$$
||A+XA^{-1}X^{*}||>||A+X^{*}A^{-1}X||.
$$

\end{lem}

After finishing this paper the author learned that 
there is a similar result in \cite{GLRT} in the case 
$B=k-A$ for some positive constant $k$.

\begin{proof} 
We set 
$
H=\begin{bmatrix}
A&X\\
X^{*}&B
\end{bmatrix}
$ and $\lambda=||H||$. 
By the assumption, 
we have $\lambda>||A+B||$. 
We can find two vectors $\xi$ and $\eta$ 
such that $||\xi||^{2}+||\eta||^{2}\not=0$ and 
$$
\begin{bmatrix}
A&X\\
X^{*}&B
\end{bmatrix}
\begin{bmatrix}
\xi\\
\eta
\end{bmatrix}=
\lambda
\begin{bmatrix}
\xi\\
\eta
\end{bmatrix}.
$$
Then we get 
$$
A\xi+X\eta=\lambda\xi,
\ \ \ \ \ \ 
X^{*}\xi+B\eta=\lambda\eta.
$$
Since $\lambda>||A+B||$, both $\lambda-A$ 
and $\lambda-B$ are invertible. 
Then we can rewrite the above relations as 
$$
(\lambda-A)^{-1}X\eta=\xi,
\ \ \ \ \ \ 
(\lambda-B)^{-1}X^{*}\xi=\eta.
$$
Therefore we get 
$$
(\lambda-A)^{-1}X(\lambda-B)^{-1}X^{*}\xi=\xi
$$
and hence 
$$
X(\lambda-B)^{-1}X^{*}\xi=\lambda\xi-A\xi.
$$ 
Thus we have 
$$
(A+X(\lambda-B)^{-1}X^{*})\xi=\lambda\xi.
$$ 
Here we remark that $\xi\not=0$. Indeed 
recall the relation 
$(\lambda-B)^{-1}X^{*}\xi=\eta$. 
By this equality, if $\xi=0$, then we must have $\eta=0$. 
This contradicts the fact $||\xi||^{2}+||\eta||^{2}\not=0$. 
So we conclude 
$$
||A+X(\lambda-B)^{-1}X^{*}||\geq\lambda>||A+B||.
$$
Since $A+B\leq \lambda$, we have $(\lambda-B)^{-1}\leq A^{-1}$. 
Thus we get 
$$
||A+XA^{-1}X^{*}||\geq\lambda>||A+B||.
$$

\end{proof}

By this lemma, we have the following. 
\begin{thm}
Let $A>0$ and $X$ be matrices. We set 
$$
\alpha=||A+X^{*}A^{-1}X||,\ \ \ \ 
\beta=||A+XA^{-1}X^{*}||
$$
Then we have the following. 
\begin{enumerate}
\item If $\alpha>\beta$, 
then for any $B\geq X^{*}A^{-1}X$ we have 
$$
\Bigl\Vert
\begin{bmatrix}
A&X\\
X^{*}&B
\end{bmatrix}
\Bigr\Vert\leq 
||A+B||. \eqno{(1)}
$$
\item 
If $\alpha<\beta$, 
then for any $C\geq XA^{-1}X^{*}$ we have 
$$
\Bigl\Vert
\begin{bmatrix}
A&X^{*}\\
X&C
\end{bmatrix}
\Bigr\Vert\leq 
||A+C||. \eqno{(2)}
$$
\item If 
$\alpha=\beta$, then for any $B\geq X^{*}A^{-1}X$ 
and $C\geq XA^{-1}X^{*}$ 
we have 
$$
\Bigl\Vert
\begin{bmatrix}
A&X\\
X^{*}&B
\end{bmatrix}
\Bigr\Vert\leq 
||A+B||,\ \ \ \  
\Bigl\Vert
\begin{bmatrix}
A&X^{*}\\
X&C
\end{bmatrix}
\Bigr\Vert\leq 
||A+C||. 
$$
\end{enumerate} 
In particular either the inequality (1) or (2) is always true. 
\end{thm}

\begin{proof}
This immediately follows from the previous lemma. 
Indeed if $B\geq X^{*}A^{-1}X$ does not satisfy the 
inequality (1), then by the lemma we have 
$\alpha<\beta$. Similarly if $C\geq XA^{-1}X^{*}$ does not satisfy the 
inequality (2), then by the lemma we have 
$\alpha>\beta$. So we have shown both (i) and (ii). 
The statement (iii) is also obvious.
\end{proof}

Next we want to consider a special case in which 
$X$ is unitary. 
\begin{prop} 
For any positive invertible matrix $A$ and any unitary $U$, we have 
$$
||
A+A^{-1}
||
\leq
||A+U^{*}A^{-1}U||.
$$ 
That is, the inequality in the problem 2 is true if $X$ is unitary. 
\end{prop}

\begin{proof}
Let $\lambda_{min}$ 
be the minimal eigenvalue of $A$. 
Consider the function 
$
f(t)=t+t^{-1}
$. Then since 
$
f'(t)=\dfrac{t^{2}-1}{t^{2}}
$, the maximum of $f(t)$ on the interval 
$
0<a\leq t\leq b
$ is given by 
$
\max\{a+a^{-1},\ \ b+b^{-1}\}
$. 
Therefore we have 
$$
||
A+A^{-1}
||=
\max\{\lambda_{min}+\lambda_{min}^{-1}
,\ \ ||A||+||A||^{-1}
\}.
$$ 
We may assume that 
$$
||
A+A^{-1}
||=
||A||+||A||^{-1}.
$$ 
Indeed, 
by setting $B=A^{-1}$, we see that 
$
||
A+A^{-1}
||=||B+B^{-1}||
$ and 
$
||A+U^{*}A^{-1}U||=||B+UB^{-1}U^{*}||
$. Moreover the spectrum of $B$ is located 
in the interval $||A||^{-1}\leq t\leq \lambda_{min}^{-1}
=||B||
$. 
Therefore if 
$
||
A+A^{-1}
||=
\lambda_{min}+\lambda_{min}^{-1}
$, then we have 
$$
||B+B^{-1}||=||
A+A^{-1}
||=
\lambda_{min}+\lambda_{min}^{-1}
=||B||+||B||^{-1}.
$$

Now we have only to show 
$$
||A||+||A||^{-1}\leq 
||A+U^{*}A^{-1}U||.
$$
Since $A\leq ||A||$, 
we have 
$
U^{*}A^{-1}U\geq ||A||^{-1}
$. Thus we get 
$$
||A+U^{*}A^{-1}U||\geq ||A+||A||^{-1}||
=||A||+||A||^{-1}.
$$
\end{proof}

In the case that $X$ is a unitary $U$, we can rewite the problem 1 
as follows. 

\noindent{\bf Problem 3.} 
For any $A>0$, any $C\geq A^{-1}$ and any unitary $U$, 
under what condition 
can we conclude  
$$\Bigl\|
\begin{bmatrix}
A&1\\
1&C
\end{bmatrix}
\Bigr\|\leq 
||
A+U^{*}CU
||?
$$

Indeed if $X$ is a unitary $U$ in problem 1, 
we see that 
$$
\Bigl\|
\begin{bmatrix}
A&U\\
U^{*}&B
\end{bmatrix}
\Bigr\|=
\Bigl\|
\begin{bmatrix}
1&0\\
0&U^{*}
\end{bmatrix}
\begin{bmatrix}
A&1\\
1&UBU^{*}
\end{bmatrix}
\begin{bmatrix}
1&0\\
0&U
\end{bmatrix}
\Bigr\|
=\Bigl\|
\begin{bmatrix}
A&1\\
1&UBU^{*}
\end{bmatrix}
\Bigr\|
$$
and 
$
||A+B||=||A+U^{*}(UBU^{*})U||
$. Thus by letting $C=UBU^{*}$ we obtain the problem 3. 

In the previous proposition we have shown 
that the inequality in the problem 3 is true in the case 
$C=A^{-1}$. 
In the same way 
we can also show that the inequality in the problem 3 is true in the case 
$C=\alpha A^{-1}$ for any scalar $\alpha\geq1$. These facts 
might suggest that the inequality in the problem 3 is true when $AC=CA$. 
However we can construct a counter example as follows.

Set 
$$
A=
\begin{bmatrix}
1&0&0\\
0&2&0\\
0&0&3
\end{bmatrix}
,\ 
C=
\begin{bmatrix}
1&0&0\\
0&\frac{1}{2}&0\\
0&0&2
\end{bmatrix}
,\ 
U=
\begin{bmatrix}
0&1&0\\
0&0&1\\
1&0&0
\end{bmatrix}.
$$
Here we remark that 
$$
A^{-1}=\begin{bmatrix}
1&0&0\\
0&\frac{1}{2}&0\\
0&0&\frac{1}{3}
\end{bmatrix}
\leq C.
$$
We observe 
$$
\|A+U^{*}CU\|=
\Bigl\|
\begin{bmatrix}
1&0&0\\
0&2&0\\
0&0&3
\end{bmatrix}
+
\begin{bmatrix}
2&0&0\\
0&1&0\\
0&0&\frac{1}{2}
\end{bmatrix}
\Bigr\|
=\frac{7}{2}
$$
Next we compute the norm 
$
\Bigl\|
\begin{bmatrix}
A&I\\
I&C
\end{bmatrix}
\Bigr\|
$. 
We observe that 
$$
\Bigl\|
\begin{bmatrix}
a&1\\
1&c
\end{bmatrix}
\Bigr\|
=
\dfrac{a+c+\sqrt{(a-c)^{2}+4}}{2}
$$ 
for any positive numbers $a$ and $c$. 
Then we see that 
$$
\Bigl\|
\begin{bmatrix}
A&1\\
1&C
\end{bmatrix}
\Bigr\|\geq 
\Bigl\|
\begin{bmatrix}
3&1\\
1&2
\end{bmatrix}
\Bigr\|=
\dfrac{5+\sqrt{5}}{2}>\frac{7}{2}=\|A+U^{*}CU\|.
$$

\ \\

Recall the following theorem due to Ando. 
\begin{thm}[Ando \cite{A}]  
The matrix $B$ is fixed.\\ 
If the implication 
$$
\begin{bmatrix}
A&B\\
B^{*}&C
\end{bmatrix}\geq0
\Longrightarrow  
||A\sharp C||\geq||B||
$$ 
is true 
for any $A\geq0,C\geq0$, then we have 
$
||B||=r(B)
$, where $r(B)$ is the spectral radius of $B$. 
($B$ is normaloid.) 
\end{thm}
Inspired by this theorem, we want to ask the following. 

\noindent{\bf Problem 4.} If 
the inequality in the problem 1 is true for any $A>0$ and 
$B\geq X^{*}A^{-1}X$, 
what can we say about $X$? 
Can we conclude that $X$ is normal?

Next we will make some observation for the problem 4. 
We can rewrite the problem 4 as follows. 

\noindent{\bf Problem 5.}
Let $D$ be positive and let $U$ be unitary. 
If 
$\Bigl\|
\begin{bmatrix}
A&D\\
D&C
\end{bmatrix}
\Bigr\|\leq ||A+U^{*}CU||
$ 
for any $A>0$ and $C\geq DA^{-1}D$, can we conclude 
$UD=DU$? 

Indeed, take a polar decomposition $X=DU$ and 
set $C=UBU^{*}$. Then we see that 
$$\Bigl\|
\begin{bmatrix}
A&X\\
X^{*}&B
\end{bmatrix}
\Bigr\|=
\Bigl\|
\begin{bmatrix}
A&DU\\
U^{*}D&U^{*}CU
\end{bmatrix}
\Bigr\|
=
\Bigl\|
\begin{bmatrix}
A&D\\
D&C
\end{bmatrix}
\Bigr\|
$$ 
and 
$
||A+B||=||A+U^{*}CU||
$. On the other hand we observe that 
the inequality 
$
B\geq X^{*}A^{-1}X=U^{*}DA^{-1}DU
$ is equivalent to 
$
C=UBU^{*}\geq DA^{-1}D
$. Therefore 
we conclude that 
the problem 4 is equivalent to the problem 5. 
Here we remark that $D=(XX^{*})^{\frac{1}{2}}$.  

\begin{lem}
Under the assumption in the problem 5, 
we have 
$$
||D+U^{*}DU||=2||D||.
$$
\end{lem}

\begin{proof}
Set $A=C=D$. Here we remark that 
$
C=D=DA^{-1}D
$. By the assumption we have 
$$
\Bigl\|
\begin{bmatrix}
D&D\\
D&D
\end{bmatrix}
\Bigr\|\leq ||D+U^{*}DU||.
$$
Then since 
$
\Bigl\|
\begin{bmatrix}
D&D\\
D&D
\end{bmatrix}
\Bigr\|
=2||D||
$, we see that 
$$
2||D||=
\Bigl\|
\begin{bmatrix}
D&D\\
D&D
\end{bmatrix}
\Bigr\|\leq ||D+U^{*}DU||
\leq2||D||.
$$
So we are done. 
\end{proof}

\begin{lem}
Under the assumption in the problem 5, 
we can find a unit vector $\xi$ satisfying both 
$
D\xi=||D||\xi
$ and 
$
DU\xi=||D||U\xi
$. 
\end{lem}

\begin{proof}
By the previous lemma, we can take a unit vector $\xi$ such that 
$$
(D+U^{*}DU)\xi=2||D||\xi. 
$$
Then since 
$$2||D||=
||(D+U^{*}DU)\xi||\leq ||D\xi||+||U^{*}DU\xi||\leq 2||D||,
$$ 
we have both 
$||D\xi||=||D||$ and $||U^{*}DU\xi||=||D||$. 
Since 
$$
||(||D||^{2}-D^{2})^{\frac{1}{2}}\xi||^{2}=
\langle
(||D||^{2}-D^{2})\xi,\xi
\rangle=
||D||^{2}-||D\xi||^{2}=0,
$$ 
we have $D\xi=||D||\xi$. Similarly we get 
$
U^{*}DU\xi=||D||\xi
$. 
\end{proof}

We have the partial answer to the problem 5 as follows. 

\begin{thm}
Under the assumption in the problem 5, if the $n\times n$-matrix $D$ 
has $n$ distinct eigenvalues, then we have $UD=DU$. 
That is, the problem 5 is true in this case. 
\end{thm} 
Here recall that $D=(XX^{*})^{\frac{1}{2}}$ and that 
the prolem 4 is equivalent to the problem 5. These mean that 
the problem 4 is true if $(XX^{*})^{\frac{1}{2}}$ has $n$ distinct eigenvalues. 

For the proof we need some preparation. 

\begin{lem}{\cite[Lemma 2.1]{AAB}}
For two positive operators $A$ and $B$, if they satisfy 
$
||A+B||=||A||+||B||
$, then we have $
||\alpha A+\beta B||=\alpha||A||+\beta||B||
$ for any $\alpha\geq0$ and $\beta\geq0$. 
\end{lem}

\begin{proof}
We would like to include its proof for completeness. Without loss of generality 
we may assume that 
$
\alpha\geq\beta\geq0
$. We see that 
\begin{align*}
||\alpha A+\beta B||&=
||\alpha(A+B)-(\alpha-\beta)B||\geq
\alpha||A+B||-(\alpha-\beta)||B||\\
&=\alpha(||A||+||B||)-(\alpha-\beta)||B||
=\alpha||A||+\beta||B||.
\end{align*}
The reverse inequality follows from the triangle inequality. 
\end{proof}

\begin{lem}
Consider the matrices as in the problem 5. 
Let $q$ be a projection with 
$Dq=qD$ and $Uq=qU$ and we set $p=1-q$. 
Then we have 
$$
||Dp+U^{*}DpU||=2||Dp||.
$$
\end{lem}

\begin{proof}
We set 
$$
A=kDp+q
$$ 
where $k$ is a positive constant. Later we will take $k$ large enough. 
By the assumption we have 
$$
\Bigl\|
\begin{bmatrix}
A&D\\
D&DA^{-1}D
\end{bmatrix}
\Bigr\|
=||A+A^{-\frac{1}{2}}D^{2}A^{-\frac{1}{2}}||
\leq ||A+U^{*}DA^{-1}DU||.
$$
We see that 
$$
||A+A^{-\frac{1}{2}}D^{2}A^{-\frac{1}{2}}||\geq 
||(A+A^{-\frac{1}{2}}D^{2}A^{-\frac{1}{2}})p||
=\Bigl(k+\dfrac{1}{k}\Bigr)||Dp||.
$$ 
Thus we conclude that 
$$
\Bigl(k+\dfrac{1}{k}\Bigr)||Dp||\leq ||A+U^{*}DA^{-1}DU||.
$$
On the other hand we observe 
\begin{align*}
 ||A+U^{*}DA^{-1}DU||&=
||
kDp+\dfrac{1}{k}U^{*}DpU+(1+U^{*}D^{2}U)q
||\\
&=
\max\{
||kDp+\dfrac{1}{k}U^{*}DpU||,\ \ 
||(1+U^{*}D^{2}U)q||
\}
\end{align*} 
because the operator $kDp+\dfrac{1}{k}U^{*}DpU$ 
is orthogonal to $(1+U^{*}D^{2}U)q$. 
(Recall that both $p$ and $q$ commute with $D$ and $U$.) 
If $Dp=0$, we have nothing to do. If $Dp\not=0$, we can take 
the constant $k>0$ large enough such that 
$$
||kDp+\dfrac{1}{k}U^{*}DpU||\geq
k||Dp||-\dfrac{1}{k}||U^{*}DpU||\geq 
||(1+U^{*}D^{2}U)q||
$$  
and hence 
$$
||A+U^{*}DA^{-1}DU||=||kDp+\dfrac{1}{k}U^{*}DpU||.
$$ 
Then we have 
$$
\Bigl(k+\dfrac{1}{k}\Bigr)||Dp||\leq ||kDp+\dfrac{1}{k}U^{*}DpU||
\leq 
||kDp||+||\dfrac{1}{k}U^{*}DpU||=(k+\dfrac{1}{k})||Dp||.
$$ 
That is, we get 
$$
||kDp+\dfrac{1}{k}U^{*}DpU||=||kDp||+||\dfrac{1}{k}U^{*}DpU||.
$$ 
By the previous lemma we have the desired statement. 
\end{proof}

\noindent{\it Proof of Theorem 2.7.} 
Since each eigenvalue of $D$ has multiplicity 1, by lemma 2.6 
there exists a rank one projection $q_{1}$ such that 
$Dq_{1}=q_{1}D=||D||q_{1}$ and $Uq_{1}=q_{1}U$. We set $p_{1}=1-q_{1}$. 
Then applying lemma 2.9 to $q_{1}$ and $p_{1}$, we obtain 
$$
||Dp_{1}+U^{*}Dp_{1}U||=2||Dp_{1}||.
$$ 
Then by the proof of lemma 2.6, we can find a unit vector $\xi=p_{1}\xi$ such that 
$
D\xi=||Dp_{1}||\xi
$ and 
$
DU\xi=||Dp_{1}||U\xi
$. Since the eigenvalue $||Dp||$ of $D$ has multiplicity 1, 
we can find a rank 1 projection $q_{2}\leq p_{1}$ such that 
$Dq_{2}=q_{2}D=||Dp_{1}||q_{2}$ and $Uq_{2}=q_{2}U$. We set $p_{2}=1-(q_{1}+q_{2})$. 
By applying 
lemma 2.9 again, we get 
$$
||Dp_{2}+U^{*}Dp_{2}U||=2||Dp_{2}||. 
$$ 
By continuing this procedure, we 
can construct mutually orthogonal rank 1 projections 
$q_{1},q_{2},\cdots,q_{n}$ such that $Uq_{j}=q_{j}U$ and 
$
D=\lambda_{1}q_{1}+\cdots+\lambda_{n}q_{n}
$, $(\lambda_{1}>\lambda_{2}>\cdots>\lambda_{n})$. 
Then we conclude that 
$
UD=DU.
$ \qed

\end{document}